\documentclass{article}

\usepackage{amsmath,amssymb,amsthm}
\usepackage[utf8x]{inputenc}
\usepackage[T1]{fontenc}
\usepackage{lmodern}
\usepackage{microtype}
\usepackage{bbding}
\usepackage{color}
\usepackage{subcaption}
\usepackage{float}
\usepackage{authblk}

\usepackage{tikz}
\usetikzlibrary{arrows}
\usetikzlibrary{arrows}
\usetikzlibrary{bending}
\usetikzlibrary{arrows}
\usetikzlibrary{positioning}

\usepackage{mathmag}
\usepackage{qedequation}

\newuntheorem{thm}{Theorem}
\newuntheorem{exer}{Exercise}
\newuntheorem{problem}{Problem}
\newuntheorem{conjecture}{Conjecture}

\newcommand{\ff}[1]{{\mathbb F}_{#1}}

\newcommand{\Z}{\mathbb Z}
\newcommand{\F}{\mathbb F}

\renewcommand{\and}{\\}

\begin{document}
\title{A result on polynomials derived via graph theory}

\author{
        Robert S. Coulter\\
            \scriptsize University of Delaware\\
            \scriptsize Newark, Delaware 19716\\
            \small\tt coulter@udel.edu
        \and 
        \vspace{.5cm}
        Stefaan De Winter\\
            \scriptsize Michigan Technological University\\
            \scriptsize Houghton, Michigan 49931\\
            \small\tt sgdewint@mtu.edu
        \and
        \vspace{.5cm}
        Alex Kodess\\
            \scriptsize Farmingdale State College\\
            \scriptsize Farmingdale, New York 11735\\
            \small\tt alex.kodess@farmingdale.edu
        \and 
        \vspace{.5cm}
        Felix Lazebnik\\
            \scriptsize University of Delaware\\
            \scriptsize Newark, Delaware 19716\\
            \small\tt fellaz@udel.edu
        }

\maketitle

\noindent
{\bf Math Subject Classification: 
    05C20, 
    11T06, 
    05C60.
}

\bigskip

Graph theory is a comparatively young mathematical discipline.  It is often hard to construct graphs that satisfy certain  properties purely combinatorially,  i.e., by taking a set of vertices and saying which vertex is connected to which.     Often such areas of classical mathematics as number theory,  geometry, or algebra  are used for this,  and the methods from the related areas are used to prove the properties of the obtained graphs.  The examples are numerous, and   many 
of them can be found in books and comprehensive survey articles.  See, for example,
 Alon \cite{alon_icm}; 
Babai and Frankl \cite{babai_frankl}; 
Biggs \cite{Biggs94}; 
F\" uredi and Simonovits \cite{FurSim13}; Brouwer and Haemers \cite{br_haem}, and 
Alon and Spencer \cite{alon_spencer}. Here we wish to mention just a few such applications. 
The probabilistic method was used to prove the 
existence of certain graphs in Ramsey theory, and explicit constructions for these graphs are still unknown (see \cite{alon_spencer}). 
Constructions and analysis of Ramanujan graphs are often based on algebra and number theory. 
Methods of linear algebra are fundamental for studies of expanders and graphs with high degree of symmetry (see \cite{babai_frankl} and \cite{br_haem}). 
Lov\'asz's proof \cite{lovasz_kneiser} of a conjecture on the chromatic number of Kneser graphs made use of algebraic topology.

 Can the direction  be reversed, i.e., can  graph theory be used to obtain results in some classical areas of mathematics?     Sometimes it can,  but the number of such instances is much smaller. See, for example,  Swan's proof of the Amitsur-Levitzki theorem \cite{Swan63}, or  
 a counterexample to Borsuk’s conjecture
 by Kahn and Kalai \cite{KahnKal93} and related work by  Bondarenko \cite{bondar}. 
 Extremal graph theory was used in probability by Katona \cite{Kat69}, and in geometry and potential theory by Tur\'an  \cite{Tur70},  and  Erd\H{o}s, Meir,  Sos, and Tur\'an \cite{ErdMeiSosTur72}.  
 For some applications of graph theory to linear algebra, see Doob \cite{Doob84}. A number of applications of graph theory to pure mathematics are mentioned in 
 Lov\'asz, Pyber, Welsh and Ziegler \cite{LovPybWelZie95}. 
The story we wish to share is about one such example. It  was discovered  entirely not by design.

In order to describe our problem,  we need a few preliminaries.   Let $p$ be a prime number,  and $\Z_p= \{0,1,2, \ldots, p-1\}$  be the set of residue classes of integers modulo $p$,  where each class is represented by the unique integer $i$,
$0\le i\le p-1$
 belonging to that class.
It is known (see for example, Ireland and Rosen \cite{ireland}) that with respect to modular arithmetic, $\Z_p$ is a field.     For instance, in $\Z_{7}$, $1 + 6 = 0$, $3\cdot 4 = 5$, and
$3^{-1} = 5$ since $3\cdot 5 = 1$.
 One can consider polynomials with coefficients from $\Z_p$;  let $\Z_p[X]$ denote the set of all of them.  Every polynomial  $f\in \Z_p[X]$ defines a function on $\Z_p$ when it is evaluated at elements of $\Z_p$. For example,   for  $f= X^3 -4X+6 = X^3 +3X+6 \in \Z_7[X]$,
$f(0)=6$, $f(1)=1^3+3\cdot 1+6=10=3$, and $f(2)=20= 6$.   Also,   $f(3) = 42=0$,  and we say that $3$ is a root of $f$.  Counting or estimating the number of roots of polynomials from $\Z_p[X]$  in $\Z_p$ is a fundamental problem in the area of mathematics called  algebraic geometry.

All results in this article hold over any finite field of odd characteristic, but for ease of presentation we will simply use the field $\mathbb{Z}_p$, $p > 2$.


Here is the problem. Recently we were surprised to learn that for any prime  $p\ge 3$ and any natural numbers $m$ and $n$ satisfying
$mn\equiv 1\bmod (p-1)$,  the trinomials
$X^{m+1}-2X+1$ and $X^{n+1}-2X+1$   have the same number of distinct roots in  $\ff{p}$.  Of course,  the coefficients 1 and $-2$ are elements of $\F_p$ and  $ -2= p-2$. For example, it is easy to check that 
for $p = 11$, $m = 3$, $n = 7$, 
the trinomial $X^{m+1} - 2X + 1$ has roots $1$, $5$, and $8$ (with root $8$ of multiplicity 2), and the trinomial $X^{n+1} - 2X + 1$ has roots 
$1$, $2$, and $3$.

 How did we arrive at this strange fact?   We will explain it a bit later,   after we discuss a special  class of digraphs.

A {\it directed graph}, or just {\it digraph},
$D = (V,A)$ is a pair of two sets $V$ and $A \subseteq V\times V$; $V$ is called
the set of {\it vertices} of $D$, and $A$ is called the set of {\it arcs} of $D$. 
All undefined terms related to digraphs can be found in Bang-Jensen and Gutin \cite{bbangjensen09}.

The digraph $D$ of Figure \ref{digraph}(a) has vertex set $V = \{a,b,c,d\}$ and
arc set $A = \{(a,b),(a,c),(b,b),(b,c),(c,a),(d,a),(d,c)\}$.
Arc $(b,b)$ is called a {\it loop},  and we say that vertex $b$ has a loop on it.
Given a
digraph $D = (V,A)$, a digraph $H = (V_1, A_1)$ is called a {\it subdigraph} of $D$ if $V_1\subseteq V$ and $A_1\subseteq A$ (see Figure~\ref{digraph}(b)).

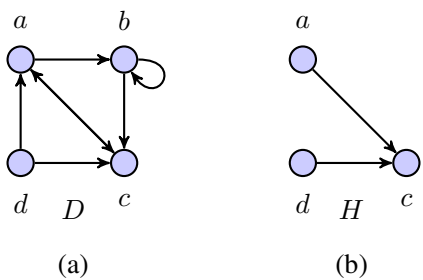
\begin{figure}[H]
    \begin{center}
        \begin{tikzpicture}[->,>=stealth',shorten >=0pt,auto,node distance=3cm,thick,
        every node/.style={circle,fill=blue!20,draw
        }]
        \node[label=above:$a$] (a) {};
        \node[label=above:$b$] (b) [right = 1cm of a] {};
        \node[label=below:$c$] (c) [below = 1cm of b] {};
        \node[label=below:$d$] (d) [below = 1cm of a] {};

        \node[draw=none,fill=none,rectangle] at (.7,-2) {$D$};

        \node[draw=none,fill=none,rectangle] at (.7,-2.75) {(a)};

        \path[every node/.style={font=\sffamily\small}]
        (a) edge node [left] {} (b);
        \path[every node/.style={font=\sffamily\small}]
        (b) edge node [left] {} (c);
        \path[every node/.style={font=\sffamily\small}]
        (d) edge node [left] {} (c);
        \path[every node/.style={font=\sffamily\small}]
        (d) edge node [left] {} (a);
        \path[every node/.style={font=\sffamily\small}]
        (a) edge node [left] {} (c);
        \path[every node/.style={font=\sffamily\small}]
        (c) edge node [left] {} (a);
         \draw (b) to [out=0,in=300,looseness=9] (b);

        \node[label=above:$a$] (a1)[right = 2 cm of b] {};
        \node[label=below:$d$] (d1) [below = 1cm of a1] {};
        \node[label=below:$c$] (c1) [right = 1cm of d1] {};

        \node[draw=none,fill=none,rectangle] at (4.4,-2) {$H$};
        \node[draw=none,fill=none,rectangle] at (4.4,-2.75) {(b)};

        \path[every node/.style={font=\sffamily\small}]
        (a1) edge node [left] {} (c1);
        \path[every node/.style={font=\sffamily\small}]
        (d1) edge node [left] {} (c1);
        \end{tikzpicture}
        \caption{Digraph $D$ and its subdigraph $H$.}
        \label{digraph}
    \end{center}
\end{figure}

We shall be interested in a certain type of digraphs. 
For any positive prime $p$,  and any positive integers $m,n$,
we define the directed graph $D(p;m,n) = (V,A)$, 
with vertex set
$V = \ff{p}\times\ff{p}$ and arc set $A$ as follows: the ordered pair of vertices
$((x_1,x_2), (y_1,y_2))$ is an arc if
\[
x_2 + y_2 = x_1^m y_1^n.
\]

We call $D(p;m,n)$ a {\it monomial digraph}.  It is known (and often referred to as Fermat's little  
theorem), that $x^p = x$ for any $x\in\ff{p}$.
It is therefore sufficient to restrict integers $m$ and $n$ in
the definition of $D(p;m,n)$  to the set $\{1,\dotso,p-1\}$.
We thus have $(p-1)^2$ digraphs $D(p; m,n)$ for every prime $p$. 

The digraphs $D(p;m,n)$ are directed analogues
(see Kodess \cite{kod14}, Kodess and Lazebnik \cite{KL2015}) 
of particular cases of a well
studied class of algebraically defined undirected graphs having many
applications, see surveys by 
Lazebnik and Woldar
\cite{laz01} and Lazebnik, Sun, and Wang \cite{lazsun16}.

Figure \ref{dig_312} shows $D(3;1,2)$. 
Note that $((2,2), (1,0))$  is an arc,  since according to the adjacency condition above,  $2+0=2^1\cdot 1^2$ in $\F_3$;
$((1,0), (2,2))$  is not  an arc,  since  $0+2\neq  1^1\cdot 2^2$ in $\F_3$; and
  vertex $(1,2)$ has a loop on it
since,
$2 + 2 = 1\cdot 1^2$ in $\ff{3}$.

\begin{figure}[H]
    \begin{center}
        \begin{tikzpicture}

        \tikzset{vertex/.style = {shape=circle,draw,inner sep=2pt,minimum size=.5em, scale = 1.0},font=\sffamily\scriptsize\bfseries}
        \tikzset{edge/.style = {->,> = triangle 45}}
        \node[vertex] (a) at  (2,0) {$(1,0)$};
        \node[vertex] (b) at  (5,0) {$(0,0)$};
        \node[vertex] (c) at  (8,0) {$(2,0)$};

        \node[vertex] (d) at  (0,-2) {$(2,1)$};
        \node[vertex] (e) at  (2,-2) {$(1,1)$};
        \node[vertex] (f) at  (8,-2) {$(2,2)$};
        \node[vertex] (g) at  (10,-2) {$(1,2)$};

        \node[vertex] (h) at  (2,-4) {$(0,2)$};
        \node[vertex] (i) at  (8,-4) {$(0,1)$};

        \draw[edge] (a) to (b);
        \draw[edge] (a) to (d);
        \draw[edge] (a) to (e);

        \draw[edge] (b) to (a);
        \draw[edge] (b) to (c);

        \draw[edge] (c) to (b);
        \draw[edge] (c) to (f);
        \draw[edge] (c) to (g);

        \draw[edge] (d) to (e);
        \draw[edge] (d) to (h);

        \draw[edge] (e) to (a);
        \draw[edge] (e) to (c);
        \draw[edge] (e) to (h);

        \draw[edge] (f) to (a);
        \draw[edge] (f) to (c);
        \draw[edge] (f) to (i);

        \draw[edge] (g) to (f);
        \draw[edge] (g) to (i);

        \draw[edge] (h) to (d);
        \draw[edge] (h) to (e);
        \draw[edge] (h) to (i);

        \draw[edge] (i) to (f);
        \draw[edge] (i) to (g);
        \draw[edge] (i) to (h);

        \path
        (b) edge [->,>={triangle 45[flex,sep=0pt]},loop,out=240,in=270,looseness=6.5] node {} (b);
        \path
        (d) edge [->,>={triangle 45[flex,sep=0pt]},loop,out=160,in=130,looseness=8] node {} (d);
        \path
        (g) edge [->,>={triangle 45[flex,sep=0pt]},loop,out=20,in=50,looseness=8] node {} (g);
        \end{tikzpicture}
        \caption{The digraph $D(3;1,2)$: $x_2+y_2 = x_1y_1^2$.}
        \label{dig_312}
    \end{center}
\end{figure}
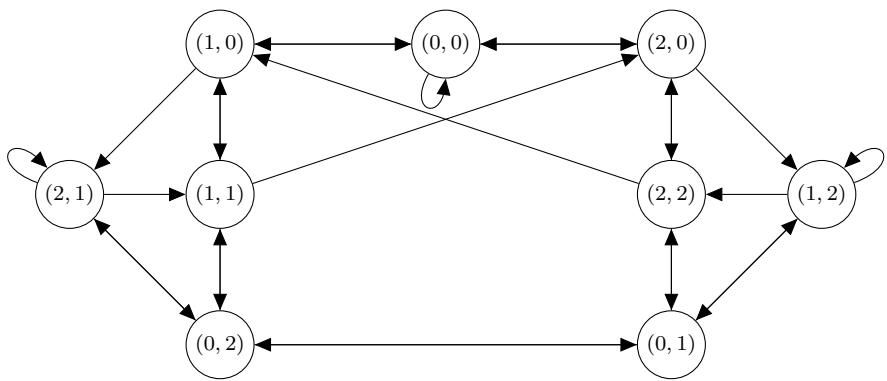

As all these $(p-1)^2$ digraphs $D(p;m,n)$ share the same vertex set, one cannot
help wondering if they are actually different. For instance,
it is not hard to see that $D(3; 1,2)$ and $D(3; 2,1)$ can be
obtained one from the other by reversing the orientation of every arc,
but {\it not} by relabeling the vertices! 
The reason for this will become more clear later. 
A very thorough and tedious inspection or any modern computer would reveal that
the digraphs $D(5; 1,2)$ and $D(5; 3,2)$, both having 25 vertices,
are in fact not much different: one can be obtained from the other
by relabeling the vertices in a certain way.

We would like to make this discussion a little more formal.

Central to many areas of mathematics is the concept of {\it isomorphism}.
It is defined for such ubiquitous and important objects as vector spaces, groups, fields and
graphs, to name just a few.
Informally, two objects are called isomorphic if they are not
fundamentally different in their structure; that is, one of them
can be obtained from the other by renaming or relabeling the elements
while preserving the internal structure.
Formally, we call digraphs $D_1$ and $D_2$ isomorphic and write $D_1 \cong D_2$
if there is a bijective function $f$
from the vertex set $V(D_1)$ of $D_1$ to the vertex set $V(D_2)$ of $D_2$ such
that for any two vertices $u,v\in V(D_1)$, $(u,v)$ is an arc in $D_1$ if and only if
$(f(u),f(v))$ is an arc in $D_2$. 
That is, $f$ preserves adjacency and non-adjacency mapping vertices of $D_1$ to those
of $D_2$. Such a mapping $f$ is called an {\it isomorphism from
$D_1$ to $D_2$}. To illustrate this idea we refer to Figure \ref{dig_iso}.
\begin{figure}[h]
    \captionsetup[subfigure]{justification=centering}
    \begin{center}
        \begin{subfigure}[b]{0.4\textwidth}
            \centering
            \hspace*{\fill}
            \begin{tikzpicture}[->,>=stealth',shorten >=0pt,auto,node distance=3cm,thick,
            every node/.style={circle,fill=blue!20,draw,font=\sffamily\small\bfseries}]
            \node[label=above:$1$] (1) {};
            \node[label=above:$2$] (2) [right = 1cm of 1] {};
            \node[label=below:$3$] (3) [below = 1cm of 2] {};
            \node[label=below:$4$] (4) [below = 1cm of 1] {};

            \node[draw=none,fill=none,rectangle] at (.7,-2) {$D_1$};

            \path[every node/.style={font=\sffamily\small}]
            (1) edge node [left] {} (2);
            \path[every node/.style={font=\sffamily\small}]
            (2) edge node [left] {} (3);
            \path[every node/.style={font=\sffamily\small}]
            (3) edge node [left] {} (4);
            \path[every node/.style={font=\sffamily\small}]
            (4) edge node [left] {} (1);
            \end{tikzpicture}
            \hfill
            \begin{tikzpicture}[->,>=stealth',shorten >=0pt,auto,node distance=3cm,thick,
            every node/.style={circle,fill=blue!20,draw,font=\sffamily\small\bfseries}]
            \node[label=above:$a$] (a) {};
            \node[label=above:$c$] (c) [right = 1cm of 1] {};
            \node[label=below:$b$] (b) [below = 1cm of 2] {};
            \node[label=below:$d$] (d) [below = 1cm of 1] {};

            \node[draw=none,fill=none,rectangle] at (0.7,-2) {$D_2$};

            \path[every node/.style={font=\sffamily\small}]
            (a) edge node [left] {} (b);
            \path[every node/.style={font=\sffamily\small}]
            (b) edge node [left] {} (c);
            \path[every node/.style={font=\sffamily\small}]
            (c) edge node [left] {} (d);
            \path[every node/.style={font=\sffamily\small}]
            (4) edge node [left] {} (1);
            \end{tikzpicture}
            \hspace*{\fill}
            \caption{Two isomorphic digraphs.}
        \end{subfigure}
        \hspace{1cm}
        \begin{subfigure}[b]{0.4\textwidth}
            \centering
            \hspace*{\fill}
            \begin{tikzpicture}[->,>=stealth',shorten >=0pt,auto,node distance=3cm,thick,
            every node/.style={circle,fill=blue!20,draw,font=\sffamily\small\bfseries}]
            \node[label=above:$1$] (1) {};
            \node[label=above:$2$] (2) [right = 1cm of 1] {};
            \node[label=below:$3$] (3) [below = 1cm of 2] {};
            \node[label=below:$4$] (4) [below = 1cm of 1] {};

            \node[draw=none,fill=none,rectangle] at (.7,-2) {$D_1$};

            \path[every node/.style={font=\sffamily\small}]
            (1) edge node [left] {} (2);
            \path[every node/.style={font=\sffamily\small}]
            (2) edge node [left] {} (3);
            \path[every node/.style={font=\sffamily\small}]
            (3) edge node [left] {} (4);
            \path[every node/.style={font=\sffamily\small}]
            (4) edge node [left] {} (1);
            \end{tikzpicture}
            \hfill
            \begin{tikzpicture}[->,>=stealth',shorten >=0pt,auto,node distance=3cm,thick,
            every node/.style={circle,fill=blue!20,draw,font=\sffamily\small\bfseries}]
            \node[label=above:$a$] (a) {};
            \node[label=above:$c$] (c) [right = 1cm of 1] {};
            \node[label=below:$b$] (b) [below = 1cm of 2] {};
            \node[label=below:$d$] (d) [below = 1cm of 1] {};

            \node[draw=none,fill=none,rectangle] at (0.7,-2) {$D_3$};

            \path[every node/.style={font=\sffamily\small}]
            (a) edge node [left] {} (b);
            \path[every node/.style={font=\sffamily\small}]
            (b) edge node [left] {} (c);
            \path[every node/.style={font=\sffamily\small}]
            (c) edge node [left] {} (d);
            \path[every node/.style={font=\sffamily\small}]
            (a) edge node [left] {} (d);
            \end{tikzpicture}
            \hspace*{\fill}
            \caption{Two non-isomorphic digraphs.}
        \end{subfigure}
    \end{center}
\caption{The concept of isomorphism of digraphs.}
\label{dig_iso}
\end{figure}
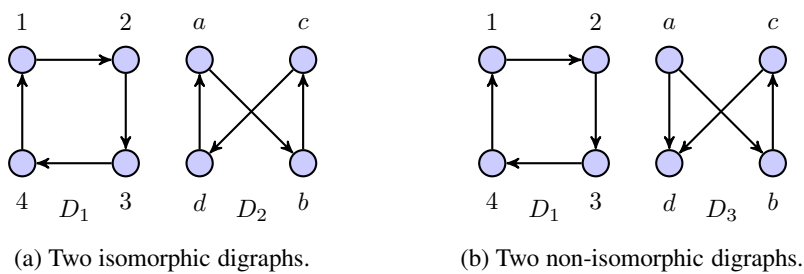

The two digraphs
on the left, $D_1$ and $D_2$, are isomorphic because the mapping
defined as $f(1) = a$, $f(2) = b$, $f(3) = c$, $f(4) = d$, is clearly a bijection; and
it is a routine verification to check that $f$ preserves adjacency and non-adjacency.
For instance, $(2,3)$ is an arc in $D_1$, and its image $(f(2),f(3)) = (b,c)$
is an arc in $D_2$, whereas $(1,3)$ and its image $(f(1),f(3)) = (a,c)$
are not arcs in $D_1$ and $D_2$, respectively.

The reader should be convinced that not only arcs but all
digraph-theoretic properties (that is, those independent of the actual labeling of the vertices)
are shared by two isomorphic digraphs.  For example,  if $g$ is an isomorphism from a  digraph $H_1$ to a digraph  $H_2$,  then every  vertex $x$  of $H_1$ and the vertex $g(x)$ of $H_2$ have the same number of in-going arcs,  and the same number of out-going arcs. This observation helps
establishing the fact that the two digraphs on
the right in Figure~\ref{dig_iso}, $D_1$ and $D_3$,
are not isomorphic: in $D_3$ vertex $a$ has two out-going
arcs, whereas $D_1$ has no vertex with this property. Other properties shared by isomorphic (di)graphs include the number of (directed) cycles of a given length, the total number of (directed) cycles, the number of (strong) components, etc. 
Note that simply reversing the orientation of every arc in a digraph 
may or may not produce a digraph isomorphic to the original one. See Figure~\ref{arc_rev}. 

\newlength{\mylength}
\setlength{\mylength}{1.118034cm} 
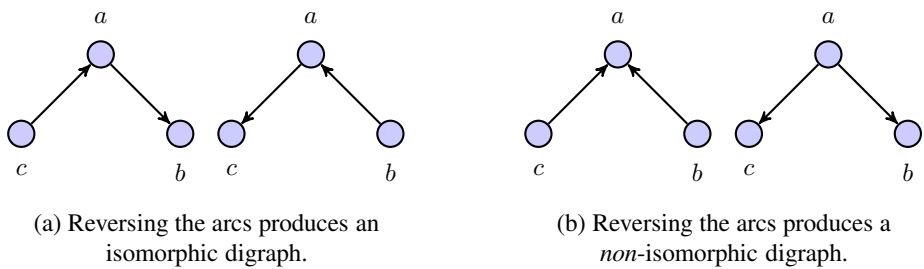
\begin{figure}[h]
    \captionsetup[subfigure]{justification=centering}
    \begin{center}
        \begin{subfigure}[b]{0.45\textwidth}
            \centering
            \hspace*{\fill}
            \begin{tikzpicture}[->,>=stealth',shorten >=0pt,auto,node distance=3cm,thick,
            every node/.style={circle,fill=blue!20,draw,font=\sffamily\small\bfseries}]
            \node[label=above:$a$] (a) {};
            \node[label=below:$b$] (b) [below right = \mylength of 1] {};
            \node[label=below:$c$] (c) [below left = \mylength of 1] {};
            
            
            \path[every node/.style={font=\sffamily\small}]
            (c) edge node [left] {} (a);
            \path[every node/.style={font=\sffamily\small}]
            (a) edge node [left] {} (b);
            \end{tikzpicture}
            \hfill
            \begin{tikzpicture}[->,>=stealth',shorten >=0pt,auto,node distance=3cm,thick,
            every node/.style={circle,fill=blue!20,draw,font=\sffamily\small\bfseries}]
            \node[label=above:$a$] (a) {};
            \node[label=below:$b$] (b) [below right = \mylength of 1] {};
            \node[label=below:$c$] (c) [below left = \mylength of 1] {};
            
            
            \path[every node/.style={font=\sffamily\small}]
            (a) edge node [left] {} (c);
            \path[every node/.style={font=\sffamily\small}]
            (b) edge node [left] {} (a);
            \end{tikzpicture}
            \hspace*{\fill}
            \caption{Reversing the arcs produces an isomorphic digraph.}
        \end{subfigure}
        \hspace{1cm}
        \begin{subfigure}[b]{0.45\textwidth}
            \centering
            \hspace*{\fill}
            \begin{tikzpicture}[->,>=stealth',shorten >=0pt,auto,node distance=3cm,thick,
            every node/.style={circle,fill=blue!20,draw,font=\sffamily\small\bfseries}]
            \node[label=above:$a$] (a) {};
            \node[label=below:$b$] (b) [below right = \mylength of 1] {};
            \node[label=below:$c$] (c) [below left = \mylength of 1] {};
            
            
            \path[every node/.style={font=\sffamily\small}]
            (b) edge node [left] {} (a);
            \path[every node/.style={font=\sffamily\small}]
            (c) edge node [left] {} (a);
            \end{tikzpicture}
            \hfill
            \begin{tikzpicture}[->,>=stealth',shorten >=0pt,auto,node distance=3cm,thick,
            every node/.style={circle,fill=blue!20,draw,font=\sffamily\small\bfseries}]
            \node[label=above:$a$] (a) {};
            \node[label=below:$b$] (b) [below right = \mylength of 1] {};
            \node[label=below:$c$] (c) [below left = \mylength of 1] {};
            
            
            \path[every node/.style={font=\sffamily\small}]
            (a) edge node [left] {} (b);
            \path[every node/.style={font=\sffamily\small}]
            (a) edge node [left] {} (c);
            \end{tikzpicture}
            \hspace*{\fill}
            \caption{Reversing the arcs produces a {\it non}-isomorphic digraph.}
        \end{subfigure}
    \end{center}
    \caption{Reversing the arcs of a digraph.}
    \label{arc_rev}
\end{figure}

Suppose one has a large set of digraphs and wants to find all  its members with a particular property.   Every member of the set can be considered and checked for having the  property, but, as isomorphic digraphs possess  the property simultaneously,   it is sufficient to check only one of them.   So only one member of a class of isomorphic digraphs can be considered.  Therefore,  if one has an efficient way for establishing isomorphism of digraphs from the set,   the original set of digraphs  can be replaced by a smaller subset of it  (and often much smaller) consisting of one ``representative'' of each  class of isomorphic graphs,  and the property can be checked only for digraphs from this subset.  This approach becomes  even more efficient when we wish to check multiple properties for digraphs from the original set.  Once its members are ``sorted for  isomorphism'',  every property can be checked for only one representative of each isomorphic class.  Unfortunately,  establishing isomorphism between  digraphs is often not easy.

Asking whether two objects are isomorphic and searching
for effective computational tools for answering this question
has been the subject of a number of highly publicized
mathematical endeavors in the 20th century. The reader may have heard of
the Classification of Finite Simple Groups problem that
seeks to give a complete list of such groups up to isomorphism;
see expositions by Solomon \cite{Sol_FSG, Sol_hist}.
Another example
is the Graph Isomorphism Problem
 which  is concerned with finding fast algorithms for  determining whether
two finite graphs are isomorphic. For recent progress on this problem, see
Babai \cite{babai}.

The question of isomorphism of two monomial digraphs
$D_1=D(p;m_1,n_1)$ and $D_2=D(p;m_2,n_2)$ is open, and it is this question
that originally motivated us. In an attempt to answer this question
one would seek necessary and sufficient conditions on the parameters
$m_1,n_1,m_2,n_2$ under which the two digraphs $D_1$ and $D_2$
are isomorphic.
One idea to attack this problem is to look at various subdigraphs
of $D_1$ and $D_2$.

Let $X$ and $Y$ be arbitrary digraphs, and let $N(X,Y)$ denote the number of subdigraphs of $X$ each of which is isomorphic to $Y$.
In trying to decide whether two given digraphs $X_1$ and $X_2$ are isomorphic
one could hope to find a ``test digraph''
$Y$ such that $N(X_1,Y) = N(X_2,Y)$ if and only if $X_1\cong X_2$.
This approach was successful in the case of a certain class 
of undirected
graphs, see
Dmytrenko, Lazebnik, and Viglione
\cite{dmy05}.
In attempting to replicate this success for the class of monomial digraphs, we were led to
consider the digraph $K$ of Figure \ref{dig_K}.
We must admit at this point that $K$ was not a good test digraph: much to
our regret, we discovered a great many pairs of non-isomorphic monomial digraphs $D_1$ and $D_2$ containing
the same number of (isomorphic) copies of $K$. Counting $N(D(p;m,n), K)$,
however, led to the 
 result on the number of roots of certain polynomials
over finite fields that we have already mentioned.  Let us  present  our solution.

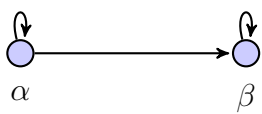
\begin{figure}[h]
\centering
\begin{tikzpicture}[->,>=stealth',shorten >=1pt,auto,node distance=3cm,thick,
every node/.style={circle,fill=blue!20,draw,font=\sffamily\Large\bfseries}]

  \node[label=below:$\alpha$] (1) {};
  \node[label=below:$\beta$] (2) [right of=1] {};

  \path[every node/.style={font=\sffamily\small}]
    (1) edge node [left] {} (2)
        edge [loop above] node {} (1)
    (2) edge [loop above] node {} ();
\end{tikzpicture}
\caption{The digraph $K$.}
\label{dig_K}
\end{figure}
\begin{thm}
For any odd prime $p$ and any natural numbers $m$ and $n$ satisfying
$mn\equiv 1\bmod (p-1)$,  the trinomials
$X^{m+1}-2X+1$ and $X^{n+1}-2X+1$   have the same number of distinct roots in $\ff{p}$.
\end{thm}

\begin{proof}
Set $D_m=D(p;1,m)$, $D_n=D(p;1,n)$ and $D_n'=D(p;n,1)$.
We first argue that $D_m$ and $D_n'$ are isomorphic, and as such, contain
the same number of isomorphic copies of $K$ shown in Figure~\ref{dig_K}.

As $mn\equiv 1\bmod (p-1)$ 
implies 
$n\cdot m - t\cdot(p-1) = 1$ for
some integer $t$, we conclude that $\gcd(m,p-1) = 1$.

We now recall that the multiplicative group $\ff{p}^{\ast}$ of $\ff{p}$ is cyclic
of order $p-1$. This implies by elementary theory of cyclic
groups that $x\mapsto x^m$ is a permutation (bijective function) on $\ff{p}$.
Also for any integers $k,l$ and any $x\in\ff{p}$,
$k\equiv l \bmod (p-1)$ implies $x^k = x^l$.
Proofs of these facts can be found in \cite{ireland}.

Consider the mapping $\psi: V(D_m) \to V(D_n')$ defined by
$\psi((x,y)) = (x^m,y)$. We verify that $\psi$ satisfies the
definition of digraph isomorphism discussed above.
Clearly $\psi$ is a permutation on
$\ff{p}^2 = V(D_m) = V(D_n')$.
Suppose $((x_1,x_2),(y_1,y_2))$ is an arc in $D_m$, that is,
\[
x_2 + y_2 = x_1 y_1^m.
\]
Then its image
$
(\psi((x_1,x_2)),\psi((y_1,y_2)))
=
((x_1^m,x_2), (y_1^m,y_2))
$
is an arc in $D_n'$ since
\[
x_2 + y_2 = x_1^1 y_1^m = x_1^{mn} y_1^m = (x_1^{m})^n (y_1^m)^1.
\]
Similarly, we show that $\psi$ preserves non-adjacency: if $((x_1,x_2),(y_1,y_2))$ is not an arc in $D_m$, then 
\[
x_2 + y_2 \neq x_1^1 y_1^m, 
\]
and so 
$
(\psi((x_1,x_2)),\psi((y_1,y_2)))$ is not 
an arc in $D_n'$. 
This implies by definition that $D_m$ and $D_n'$ are isomorphic.
Hence, $N(D_m, Y) = N(D_n', Y)$ for any digraph $Y$. 
In particular, $N(D_m, K) = N(D_n', K)$.

Now let $H^c$ denote the {\it converse} of digraph $H$, that is, the digraph
obtained from $H$ by reversing all its arcs.
Obviously, for any digraph $D$, $N(D,H) = N(D^c, H^c)$,
and also $(H^c)^c = H$.
Observe that $D_n'$ is simply $D_n$ with all arcs reversed, that is
$D_n' = D_n^c$.
Thus $D_n'^c$ and $(D_n^c)^c = D_n$ are equal.
Since $K^c\cong K$, we have
\[
N(D_m, K)=N(D_n', K)=N(D_n'^c, K^c) = N(D_n'^c, K) = N(D_n,K).
\]
Note that we did not assume that $D_m$ and $D_n$ were isomorphic! Actually
we conjecture that they never are unless $m = n$.

We now show that the number of isomorphic copies of $K$ in a digraph $D_n$
can be expressed as the number of distinct roots of a polynomial of degree $n+1$ in the field
$\ff{p}$.

Suppose $K$ is a subgraph of $D_n=D(p;1,n)$,  and let $\alpha =(u,s)$
and $\beta = (v,t)$  be vertices of $K$. 
From the relations defining the three arcs of $K$, we have
$$s+s= u\cdot u^n,\;\; t+t= v\cdot v^n,\;\;\text{and}\;\; s+t=uv^n. $$
Note that since $p$ is odd, $2\in\ff{p}$ is invertible. Hence,  we obtain 

\begin{equation}
\label{digraphKadj}
s=\frac{1}{2}u^{n+1},\;\; 
t=\frac{1}{2}v^{n+1},\;\; \text{and} \;
s+t=uv^n.
\end{equation}

If $u=v$, then  the first and the second equation of system 
 (1) imply  $s=t$,  and so vertices $\alpha$ and $ \beta$ are equal.  Therefore,  $u\ne v$.

It follows from (1) that the equation $s+t=uv^n$ can be rewritten as
\begin{equation}\label{uv}
\frac{1}{2}u^{n+1} + \frac{1}{2}v^{n+1} =uv^n.
   \end{equation}
   Note that neither $u$ nor $v$ is $0$. Indeed,  if $u=0$, then substituting it to the first and to the third equation of the system 
   (1), we get $s=0$ and $s+t=0$. Hence, $t=0$, and from the second equation we get $v=0$. Hence, $\alpha=\beta =(0,0)$ --- a contradiction.  Therefore, $u\ne 0$.  Similarly, $v\ne 0$, and $uv\ne 0$.   Dividing  both sides of
   (2) by $(1/2)uv^n$, we obtain  $$(u/v)^{n} + (v/u) = 2.$$  Setting $w=u/v$, we rewrite this equation as $w^{n+1}-2w+1=0$.
Hence,  $u/v$ is a root of the polynomial $f_n(X)=X^{n+1}-2X+1\in\mathbb{F}_p[X]$. 
It is not equal to the obvious root $1$, as $u\ne v$.

Consequently, $N(D_n,K)=(p-1)R(n)$, where $R(n)$ is the number of distinct roots of $f_n$
in $\ff{p}\setminus\{1\}$; any choice of root and any choice of $u$ must
determine $\alpha$ and $\beta$ uniquely.

Now if $m$ and $n$ are integers satisfying the conditions of the theorem, we have
from before that
\[
(p-1) R(m) = N(D_m, K) = N(D_n, K) = (p-1) R(n),
\]
and so $R(m) = R(n)$.
\qed
\end{proof}

Thus, from an isomorphism problem for digraphs, we have arrived at an interesting fact
concerning trinomials over finite fields.
The theorem can be immediately generalized in various ways and proved directly, i.e., without considering graphs or digraphs. We suggest that the reader find a proof for the following generalization.
\begin{exer}
For any prime power $q$ (even or odd) and any natural numbers $m$ and $n$ satisfying
$mn\equiv 1\bmod (q-1)$,
polynomials $X^{m+1} + aX + b$ and $X^{n+1}+aX+b^m$ have the same number of distinct 
roots in the finite field $\ff{q}$ for any $a,b\in\ff{q}$.
\end{exer}

We end this note with two open
questions concerning monomial digraphs which we find very interesting. Though we do not know the answers even for prime $q$, we state the questions for any prime power $q$.
Let $D_1 = D(q; m_1,n_1)$ and $D_2=D(q; m_2,n_2)$.

\begin{problem}
Is there a digraph $H$ 
such that the equality
$N(D_1, H) = N(D_2, H)$ 
is equivalent to
$D_1\cong D_2$?
\end{problem}

\begin{problem}
Find necessary and sufficient conditions on $q,m_1,n_1, m_2, n_2$, such that
digraphs $D_1$ and $D_2$ are isomorphic. 
\end{problem}
A related conjecture appears in \cite{kod14}:
\begin{conjecture}
\label{main_conj}
Let $q$ be a prime power, and let $m_1,n_1,m_2,n_2$ be integers from $\{1,2,\ldots, ,q-1\}$.
Then  $D(q; m_1,n_1)\cong D(q; m_2,n_2)$  if and only if there exists an integer $k$,
coprime with $q-1$ such that
\[
m_2 \equiv k m_1 \mod (q-1) \; \;\text{and}\; \;
n_2 \equiv k n_1 \mod (q-1).
\]
\end{conjecture}

\bigskip
\noindent
{\sc 
Acknowledgement}. 
The authors are thankful to the anonymous referees whose thoughtful comments
improved the paper. The work of the last author 
was partially supported by a grant from the
Simons Foundation \#426092.


\noindent
{\bf Summary} 
We present an example of a result in graph theory that is used 
to obtain a result in another branch of mathematics. More precisely, 
we show that the isomorphism of certain directed graphs implies that 
some trinomials over finite fields have the same number of roots.

\bigskip

\noindent
{\bf \uppercase{Robert S. Coulter}} (615822) (coulter@udel.edu) joined the faculty at the University of Delaware in 2003, having
previously held positions at the University of Queensland, Queensland University of Technology, and Deakin
University. He received his Ph.D. from the Department of Computer Science and Electrical Engineering at the
University of Queensland in 1998. He is an Australian, and greatly misses Farmers Union Iced Coffee and the
lack of snow shovels.

\bigskip

\noindent
{\bf\uppercase{Stefaan De Winter}} (723701) (sgdewint@mtu.edu) moved to Michigan Technological University in 2011 and
made such a great impression that he was promoted early to Associate Professor. Though he heralds from
Belgium, he abandoned the great chocolate and beer of his homeland, not to mention the cycling, and moved
to the United States in pursuit of happiness, a pursuit in which he was ironically successful through the medium
of another foreign national! 

\bigskip

\noindent
{\bf\uppercase{Alex Kodess}} (886420) (alex.kodess@farmingdale.edu) has recently moved 
to the Empire State and joined the Mathematics Department of Farmingdale State College. In the past he escaped the clutches of the state institution of the
second smallest state in the country, only to be subsumed by the smallest. In exchanging the University of
Rhode Island for the University of Delaware, he could at least be content in the knowledge his status ``improved''
from Ph.D. student to faculty member. His only regret is that he now finds he is expected to act like a responsible adult.

\bigskip

\noindent
{\bf\uppercase{Felix Lazebnik}} (111260) (fellaz@udel.edu) has been at the University of Delaware since receiving his Ph.D. from
the University of Pennsylvania under Herbert S. Wilf in 1987. He claims to understand Robert’s English,
Alex' Ph.D. Thesis and some of Stefaan’s geometry.

\end{document}